\documentclass[10pt]{amsart}

\usepackage{ amsmath,amssymb,verbatim,graphicx,amsthm}
\usepackage[top=1in, bottom=1in, left=1in, right=1in]{geometry}
\usepackage{cite}
\usepackage[sc]{mathpazo}
\usepackage[T1]{fontenc}

\newcommand{\vertiii}[1]{{\left\vert\kern-0.25ex\left\vert\kern-0.25ex\left\vert #1 
    \right\vert\kern-0.25ex\right\vert\kern-0.25ex\right\vert}}

\newcommand{\RR}{\mathbb{R}}
\newcommand{\CC}{\mathbb{C}}
\newcommand{\ZZ}{\mathbb{Z}}
\newcommand{\QQ}{\mathbb{Q}}

\theoremstyle{plain}

\newtheorem{theorem}{Theorem}[section]

\newtheorem{lemma}[theorem]{Lemma}

\usepackage{setspace}
\begin{document}

\title{Extension of Wiener-Wintner double recurrence theorem to polynomials}

\author{Idris Assani}
\address{Department of Mathematics, The University of North Carolina at Chapel Hill, 
Chapel Hill, NC 27599}
\email{assani@math.unc.edu}
\urladdr{http://www.unc.edu/math/Faculty/assani/} 

\author{Ryo Moore}
\address{Department of Mathematics, The University of North Carolina at Chapel Hill, 
Chapel Hill, NC 27599}
\email{ryom@live.unc.edu}
\urladdr{http://ryom.web.unc.edu} 

\begin{abstract}
We extend the authors' previous work on Wiener-Wintner double recurrence theorem to the case of polynomials.
\end{abstract}

\maketitle
\section{Introduction}
The following extension on Bourgain's pointwise result on double recurrence \cite{BourgainDR} was proven in \cite{WWDR}.

\begin{theorem}\label{WWDR}
Let $(X, \mathcal{F}, \mu, T)$ be a standard ergodic dynamical system, $a, b \in \ZZ$ such that $a \neq b$, and $f_1, f_2 \in L^\infty(X)$. Let 
\[W_N(f_1, f_2,x ,t) = \frac{1}{N} \sum_{n=0}^{N-1} f_1(T^{an}x)f_2(T^{bn}x) e^{2\pi i n t}. \]
\begin{enumerate}
\item (Double Uniform Wiener-Wintner Theorem) If either $f_1$ or $f_2$ belongs to $\mathcal{Z}_2^\perp$, then there exists a set of full measure $X_{f_1 \otimes f_2}$ such that for all $x \in X_{f_1 \otimes f_2}$,
\begin{equation*}\label{UniformDWW} \limsup_{N \to \infty} \sup_{t \in \RR} \left|W_N(f_1, f_2, x, t) \right| = 0 \end{equation*}
\item (General Convergence) If $f_1, f_2 \in \mathcal{Z}_2$, then for $\mu$-a.e. $x \in X$, $W_N(f_1, f_2, x, t)$ converges for all $t \in \RR$.
\end{enumerate}
\end{theorem}
One of the estimates that was established to prove the uniform convergence result above was the following (this was obtained in the  proof of Theorem 5.1 of \cite{WWDR}):
\begin{theorem}\label{Thm-estimate}
Let $(X, \mathcal{F}, \mu, T)$ be a standard ergodic dynamical system, and $f_1, f_2 \in L^\infty(X)$. If $\vertiii{\cdot}_{k+1}$ denotes the $k$-th Gowers-Host-Kra seminorm \cite{Gowers,HostKraNEA}, then
\begin{equation}\label{estimate}
\int \limsup_{N \to \infty} \sup_{t \in \RR} \left| \frac{1}{N} \sum_{n=1}^N f_1(T^{an}x) f_2(T^{bn}x) e^{2\pi i n t} \right|^2 d\mu(x) \lesssim_{a, b} \min \left\{ \vertiii{f_1}_3^2, \vertiii{f_2}_3^2 \right\}
\end{equation}
\end{theorem}
In this paper, we extend this Wiener-Wintner average to a  polynomial Wiener-Wintner, i.e. If $p(n)$ is a degree-$k$ polynomial with real coefficients, then we investigate the convergence of the average
\[ \frac{1}{N} \sum_{n=0}^{N-1} f_1(T^{an}x)f_2(T^{bn}x) e^{2\pi i p(n) t}. \]
Some study has been done on polynomial Wiener-Winter averages. For instance, E. Lesigne showed that for any ergodic system $(X, \mathcal{F}, \mu, T)$, $f \in L^1(X)$, there exists a set of full-measure $X_f$ such that for every $x \in X_f$, the average
\begin{equation}\label{polyWW} \frac{1}{N} \sum_{n=1}^N \phi(p(n)) f(T^nx) \end{equation}
converges for every continuous function $\phi: \mathbb{T} \to \CC$ and for every polynomial $p(n)$ \cite{Lesigne90, Lesigne93}. Furthermore, Lesigne showed that if $T$ is totally ergodic and $f$ belongs to the orthogonal complement of the $k$-th Abramov factor, then $(\ref{polyWW})$ converges to $0$. In \cite{Fran06},  N. Frantzikinakis extended Lesigne's result by obtaining a uniform Wiener-Wintner result i.e. $f_1$ belongs to the orthogonal complement of the $k$-th Abramov factor if and only if for a.e. $x \in X$,
\begin{equation} \lim_{N \to \infty} \sup_{p \in \RR_k[t]} \left| \frac{1}{N} \sum_{n=1}^N \phi(p(n))f(T^nx) \right| = 0, \end{equation}
where $\RR_k[t]$ denotes the set of all $k$-th degree polynomials with real coefficients, and $T$ is a totally ergodic transformation. Furthermore, Frantzikinakis showed that if $T$ is \textit{not} totally ergodic, then one cannot find a set of full measure such that for any $\phi$ and $p$, the averages in $(\ref{polyWW})$ converges to $0$. This tells us that Abramov's factors are not characteristic factors for Wiener-Wintner averages with polynomials $p$ and continuous functions $\phi$.

Recently, T. Eisner and B. Krause announced uniform Wiener-Wintner results for averages with weights involving Hardy functions and for "twisted" polynomial ergodic averages \cite{EisnerKrause}.

Note that in the works of Lesigne and Frantzikinakis, the supremum were taken over the set of all the $k$-th degree polynomials with real coefficients. We are interested in taking a supremum over $t$ on the averages
\[ \frac{1}{N} \sum_{n=1}^N f_1(T^{an}x)f_2(T^{bn}x) e^{2\pi i p(n) t}. \]
Note that, by the simple application of the spectral theorem, this establishes the norm convergence of the return time averages: There exists a set of full measure $X_{f_1, f_2, p}$ such that any $x \in X_{f_1, f_2, p}$,
\[ \frac{1}{N} \sum_{n=1}^N f_1(T^{an}x)f_2(T^{bn}x)g \circ S^{p(n)} \]
converges in $L^1(\nu)$ for any measure-preserving system $(Y, \mathcal{G}, \nu, S)$ and $g \in L^1(\nu)$.
\begin{theorem}\label{mainThm}
Let $(X, \mathcal{F}, \mu, T)$ be a standard ergodic dynamical system, $a, b \in \ZZ$ such that $a \neq b$, and $f_1, f_2 \in L^2(X)$. Suppose $p(n)$ is a degree-$k$ polynomial with real coefficients, where $k \geq 2$. Let 
\[W_N(f_1, f_2,x ,p, t) = \frac{1}{N} \sum_{n=0}^{N-1} f_1(T^{an}x)f_2(T^{bn}x) e^{2\pi i p(n) t}. \]
\begin{enumerate}
\item If either $f_1$ or $f_2$ belongs to $\mathcal{Z}_{k+1}^\perp$, then there exists a set of full measure $X_{f_1 \otimes f_2, p}$ such that for all $x \in X_{f_1 \otimes f_2, p}$,
\begin{equation*}\label{UniformDWWP} \limsup_{N \to \infty} \sup_{t \in \RR} \left|W_N(f_1, f_2, x,p, t) \right| = 0 \end{equation*}
\item If $f_1, f_2 \in \mathcal{Z}_{k+1}$, then for $\mu$-a.e. $x \in X$, $W_N(f_1, f_2, x, p, t)$ converges for all $t \in \RR$.
\end{enumerate}
\end{theorem}
To prove this result, we first consider the case either $f_1$ or $f_2$ is in $\mathcal{Z}_{k+1}^\perp$, and prove the uniformity result by applying the induction on the degree of the polynomial $p$. Our base case is covered by our previous result Theorem $\ref{WWDR}$. Then we assume that both $f_1$ and $f_2$ belong to $\mathcal{Z}_{k+1}$. For this case, we break further into two cases: One where $t \in \QQ$, and $t \notin \QQ$. For the first case, we find appropriate skew product transformation on $\mathbb{T}^2$ for each degree of $p$, and use induction to prove the result. For the second case, we use the structure of nilsystems and Liebman's polynomial convergence result \cite{Leibman} to prove the claim.

One of the main inequalities used to prove (1) of Theorem $\ref{mainThm}$ is the van der Corput's lemma, which is stated as follows (a proof can found in \cite{KuipersNiederreiter}):
\begin{lemma}[van der Corput] \label{lem-vdc}
If $(a_n)$ is a sequence of complex numbers and if $H$ is an integer between $0$ and $N-1$, then
\begin{align}\label{vdc}
\left| \frac{1}{N} \sum_{n=0}^{N-1} a_n \right|^2 
&\leq \frac{N+H}{N^2(H+1)} \sum_{n=0}^{N-1} |a_n|^2 \\
&+ \frac{2(N+H)}{N^2(H+1)^2} \sum_{h=1}^H (H + 1 - h) Re \left( \sum_{n=0}^{N-h-1} a_n \overline{a}_{n+h} \right). \nonumber
\end{align}
\end{lemma}

Throughout this paper, we assume $f_1$ and $f_2$ to be real-valued functions, and $\| f_1 \|_\infty, \| f_2 \|_\infty \leq 1$.
\section{Case when either $f_1$ or $f_2$ is in $\mathcal{Z}_{k+1}^\perp$}
Here we prove (1) of Theorem $\ref{mainThm}$. We do so by extending Theorem $\ref{Thm-estimate}$ to polynomials with higher degree and higher order Host-Kra-Ziegler factors \cite{HostKraNEA, Ziegler}.
\begin{theorem}\label{uniformPoly}
Let $(X, \mathcal{F}, \mu, T)$ be an ergodic system, and $f_1, f_2 \in L^\infty(\mu)$. Suppose $p$ is a $k$-th degree polynomial with real coefficients, where $k \geq 1$. If either $f_1$ or $f_2$ belongs to $\mathcal{Z}_{k+1}^\perp$, then there exists a set of full measure $X_{f_1 \otimes f_2, p} \subset X$ such that for all $x \in X_{f_1 \otimes f_2, p}$,
\begin{equation}\label{polynomial} \lim_{N \to \infty} \sup_{t \in \RR} \left| \frac{1}{N} \sum_{n=1}^N f_1(T^{an}x) f_2(T^{bn}x) e^{2\pi i p(n) t} \right| = 0. \end{equation}
\end{theorem}

We prove this by applying induction on $k$. Before we prove the result, we would like to prove the case for $k = 2$ directly to provide the main idea of the proof. Let $p(n) = \alpha n^2 + \beta n + \gamma$, where $\alpha, \beta, \gamma$ are real constants such that $\alpha \neq 0$. The main idea of the proof for this case is to apply the van der Corput's inequality to reduce the degree of the polynomial (to a degree-one polynomial), and use the inequality $(\ref{estimate})$ after integration. This allows us to find an upper bound for the integral of the left-hand side of $(\ref{polynomial})$ in terms of the fourth Gowers-Host-Kra seminorm of $f_1$, i.e. $\vertiii{f_1}_4$.

We apply the inequality ($\ref{vdc}$) by setting $a_n = f_1(T^{an}x) f_2(T^{bn}x) e^{2\pi i p(n) t}$, taking limit supremum and applying the Cauchy-Schwarz inequality would give us
\begin{align*}
& \limsup_{N \to \infty} \sup_{t \in \RR} \left| \frac{1}{N} \sum_{n=1}^N f_1(T^{an}x)f_2(T^{bn}x) e^{2\pi i p(n) t} \right|^2
\leq \frac{2}{(H+1)} \\
&+ \frac{4}{(H+1)^2} \sum_{h=1}^H (H + 1 - h)\limsup_{N \to \infty} \sup_{t \in \RR} \Re \left( e^{-2\pi i (\alpha h^2 + \beta h)t} \frac{1}{N} \sum_{n=1}^{N- h - 1} f_1\cdot f_1 \circ T^{ah}(T^{an}x)f_2\cdot f_2 \circ T^{bh}(T^{bn}x) e^{2\pi i (-2\alpha h n)t} \right) \\
&\leq \frac{2}{H+1} + \frac{4}{H+1} \sum_{h=1}^H  \limsup_{N \to \infty}\sup_{t \in \RR} \left| \frac{1}{N} \sum_{n=1}^{N- h - 1} f_1\cdot f_1 \circ T^{ah}(T^{an}x)f_2\cdot f_2 \circ T^{bh}(T^{bn}x) e^{2\pi i (-2\alpha h n)t} \right| \\
&\leq \frac{2}{H+1} + \left(\frac{4}{H+1} \sum_{h=1}^H  \limsup_{N \to \infty} \sup_{t \in \RR} \left| \frac{1}{N} \sum_{n=1}^{N- h - 1} f_1\cdot f_1 \circ T^{ah}(T^{an}x)f_2\cdot f_2 \circ T^{bh}(T^{bn}x) e^{2\pi i (-2\alpha h n)t} \right|^2 \right)^{1/2} 
\end{align*}
If we take the integral on both sides of the inequality on $N$, and apply H\"{o}lder's inequality, then we obtain
\begin{align}\label{totalIneq_n-sq}
& \int \limsup_{N \to \infty}\sup_{t \in \RR} \left| \frac{1}{N} \sum_{n=1}^N f_1(T^{an}x)f_2(T^{bn}x) e^{2\pi i p(n) t} \right|^2 d\mu(x)  \leq \frac{2}{H+1} \nonumber\\
&+ \left(\frac{4}{H+1} \sum_{h=1}^H  \int \limsup_{N \to \infty} \sup_{t \in \RR} \left| \frac{1}{N} \sum_{n=1}^{N- h - 1} f_1\cdot f_1 \circ T^{ah}(T^{an}x)f_2\cdot f_2 \circ T^{bh}(T^{bn}x) e^{2\pi i (-2\alpha h n)t} \right|^2 d\mu \right)^{1/2} 
\end{align}
By ($\ref{estimate}$), we know that for all $1 \leq h \leq H$, the following estimate holds for the integral on the right hand side:
\begin{align} \label{estimate_n-sq_prim} &\int \limsup_{N \to \infty} \sup_{t \in \RR} \left| \frac{1}{N} \sum_{n=1}^{N- h - 1} f_1\cdot f_1 \circ T^{ah}(T^{an}x)f_2\cdot f_2 \circ T^{bh}(T^{bn}x) e^{2\pi i (-2\alpha h n)t} \right|^2 d\mu \nonumber\\
&\lesssim_{a, b} \min\left( \vertiii{f_1 \cdot f_1 \circ T^{ah}}_3^2,\vertiii{f_2 \cdot f_2 \circ T^{bh}}_3^2 \right). \end{align}
Therefore, together with $(\ref{totalIneq_n-sq})$, $(\ref{estimate_n-sq_prim})$, and the Cauchy-Schwarz inequality, we obtain
\begin{align*}
&\int \limsup_{N \to \infty}\sup_{t \in \RR} \left| \frac{1}{N} \sum_{n=1}^N f_1(T^{an}x)f_2(T^{bn}x) e^{2\pi i p(n) t} \right|^2 d\mu(x) \\
&\lesssim_{a, b} \frac{1}{H} +  \min\left\{ \left( \frac{1}{H} \sum_{h=1}^H\vertiii{f_1 \cdot f_1 \circ T^{ah}}_3^2 \right)^{1/2}, \left( \frac{1}{H} \sum_{h=1}^H\vertiii{f_2 \cdot f_2 \circ T^{bh}}_3^2 \right)^{1/2}\right\} \\
&\leq \frac{1}{H} +   \min \left\{\left( \frac{1}{H} \sum_{h=1}^H \vertiii{f_1 \cdot f_1 \circ T^{ah}}_3^4\right)^{1/4} , \left(\frac{1}{H} \sum_{h=1}^H\vertiii{f_2 \cdot f_2 \circ T^{bh}}_3^4\right)^{1/4}\right\} ,
\end{align*}
and if we let $H \to \infty$, we obtain
\begin{equation}\label{estimateZ3} \int \limsup_{N \to \infty}\sup_{t \in \RR} \left| \frac{1}{N} \sum_{n=1}^N f_1(T^{an}x)f_2(T^{bn}x) e^{2\pi i p(n) t} \right|^2 d\mu(x) \lesssim_{a, b} \min \left\{ \vertiii{f_1}_4^2, \vertiii{f_2}_4^2 \right\}. \end{equation}
So if either $f_1$ $f_2$ is in $\mathcal{Z}_3^\perp$, we know that either $\vertiii{f_1}_4^2 = 0$ or $\vertiii{f_2}_4^2 = 0$, so for $\mu$-a.e. $x \in X$
\[\limsup_{N \to \infty} \sup_{t \in \RR} \left|  \frac{1}{N} \sum_{n=1}^N f_1(T^{an}x)f_2(T^{bn}x)e^{2\pi i p(n) t} \right|= 0. \]

One of the key observations in this example is the estimate we obtained in ($\ref{estimateZ3}$). This inequality in fact holds for any $k$-th degree polynomials.

\begin{lemma}\label{estimateLemma}
Let $(X, \mathcal{F}, \mu, T)$ be an ergodic system, and $f_1, f_2 \in L^\infty(X)$. Suppose $p(n)$ is a degree-$k$ polynomial. Then for any $a, b \in \ZZ$, we have
\begin{equation} \label{estimateDeg-kPoly}
 \int \limsup_{N \to \infty}\sup_{t \in \RR} \left| \frac{1}{N} \sum_{n=1}^N f_1(T^{an}x)f_2(T^{bn}x) e^{2\pi i p(n) t} \right|^2 d\mu(x) \lesssim_{a, b} \min \left\{ \vertiii{f_1}_{k+2}^{2}, \vertiii{f_2}_{k+2}^{2}\right\}
\end{equation}
\end{lemma}
The idea of the proof is as follows: We use an induction on $k$, the degree of the polynomial $p(n)$. The base case given in Theorem $\ref{Thm-estimate}$. For the inductive step, we assume the statement is true for $k = 1, 2, \ldots, l$, and will show that the statement holds for $k = l+1$ as well. To do so, we apply the van der Corput's inequality to reduce the degree of polynomial (just like we did for the case of degree-2 polynomials), and apply the estimate using the inductive hypothesis. This allows us to obtain the appropriate Gowers-Host-Kra seminorm for our upper bound.
\begin{proof}[Proof of Lemma $\ref{estimateLemma}$]
We proceed by induction on $k$. The base case $k = 1$ is clear from $(\ref{estimate})$. Now suppose the claim holds for $k = 1, 2, \ldots, l$. Let $p(n)$ be a polynomial with degree $l+1$. If $q_h(n) = p(n+h)-p(n)$, then $q_h(n)$ is a polynomial of degree less than or equal to $l$ for all $h$, viewing $n$ as the variable. By the inequality ($\ref{vdc}$), and the Cauchy-Schwarz inequality, we know that
\begin{align*}
& \limsup_{N \to \infty}\sup_{t \in \RR} \left| \frac{1}{N} \sum_{n=1}^N f_1(T^{an}x)f_2(T^{bn}x) e^{2\pi i p(n) t} \right|^2 \\
&\leq \frac{2}{H+1} + \frac{4}{H+1} \sum_{h=1}^H  \limsup_{N \to \infty}\sup_{t \in \RR} \left| \frac{1}{N} \sum_{n=1}^{N- h - 1} f_1\cdot f_1 \circ T^{ah}(T^{an}x)f_2\cdot f_2 \circ T^{bh}(T^{bn}x) e^{2\pi i q_h(n)t} \right| \\
&\leq \frac{2}{H+1} + \left(\frac{4}{H+1} \sum_{h=1}^H \limsup_{N \to \infty} \sup_{t \in \RR} \left| \frac{1}{N} \sum_{n=1}^{N- h - 1} f_1\cdot f_1 \circ T^{ah}(T^{an}x)f_2\cdot f_2 \circ T^{bh}(T^{bn}x) e^{2\pi i q_h(n)t} \right|^2 \right)^{1/2} 
\end{align*}
By taking limit supremum, integrating both sides, and applying the Cauchy-Schwarz inequality, we have
\begin{align*}
& \int \limsup_{N \to \infty}\sup_{t \in \RR} \left| \frac{1}{N} \sum_{n=1}^N f_1(T^{an}x)f_2(T^{bn}x) e^{2\pi i p(n) t} \right|^2 d\mu(x) \\
&\leq \frac{2}{H+1} + \left(\frac{4}{H+1} \sum_{h=1}^H  \int \limsup_{N \to \infty} \sup_{t \in \RR} \left| \frac{1}{N} \sum_{n=1}^{N- h - 1} f_1\cdot f_1 \circ T^{ah}(T^{an}x)f_2\cdot f_2 \circ T^{bh}(T^{bn}x) e^{2\pi i q_h(n)t} \right|^2 d\mu \right)^{1/2}.
\end{align*}
For any $1 \leq h \leq H$, the inductive hypothesis tells us that
\begin{align*}
&\int \limsup_{N \to \infty} \sup_{t \in \RR} \left| \frac{1}{N} \sum_{n=1}^{N- h - 1} f_1\cdot f_1 \circ T^{ah}(T^{an}x)f_2\cdot f_2 \circ T^{bh}(T^{bn}x) e^{2\pi i q_h(n)t} \right|^2 d\mu \\
&\lesssim_{a, b} \min \left\{ \vertiii{f_1 \cdot f_1 \circ T^{ah}}_{l+2}^{2}, \vertiii{f_2 \cdot f_2\circ T^{bh}}_{l+2}^{2} \right\}. 
\end{align*}
Therefore,
\begin{align*}
&\int \limsup_{N \to \infty}\sup_{t \in \RR} \left| \frac{1}{N} \sum_{n=1}^N f_1(T^{an}x)f_2(T^{bn}x) e^{2\pi i p(n) t} \right|^2 d\mu(x) \\
&\lesssim_{a, b} \frac{1}{H} + \min \left\{ \left( \frac{1}{H} \sum_{h=1}^H \vertiii{f_1 \cdot f_1 \circ T^{ah}}_{l+2}^{2} \right)^{1/2}, \left( \frac{1}{H} \sum_{h=1}^H \vertiii{f_2 \cdot f_2 \circ T^{bh}}_{l+2}^{2} \right)^{1/2} \right\} \\
&\lesssim_{a, b} \frac{1}{H} + \min \left\{ \left( \frac{1}{H} \sum_{h=1}^H \vertiii{f_1 \cdot f_1 \circ T^{ah}}_{l+2}^{2^{l+2}} \right)^{2^{-(l+2)}}, \left( \frac{1}{H} \sum_{h=1}^H \vertiii{f_2 \cdot f_2 \circ T^{bh}}_{l+2}^{2^{l+2}} \right)^{2^{-(l+2)}} \right\}
\end{align*}
and if we let $H \to \infty$, we obtain
\begin{equation*} \int \limsup_{N \to \infty}\sup_{t \in \RR} \left| \frac{1}{N} \sum_{n=1}^N f_1(T^{an}x)f_2(T^{bn}x) e^{2\pi i p(n) t} \right|^2 d\mu(x) \lesssim_{a, b} \min \left( \vertiii{f_1}_{l+3}^{2}, \vertiii{f_2}_{l+3}^{2}\right).\end{equation*}
\end{proof}

\begin{proof}[Proof of Theorem $\ref{uniformPoly}$]
Since either $f_1 \in \mathcal{Z}_{k+1}^\perp$ or $f_2 \in \mathcal{Z}_{k+1}^\perp$, the right hand side of the inequality ($\ref{estimateDeg-kPoly}$) is $0$, so we must have
\[  \limsup_{N \to \infty}\sup_{t \in \RR} \left| \frac{1}{N} \sum_{n=1}^N f_1(T^{an}x)f_2(T^{bn}x) e^{2\pi i p(n) t} \right| = 0. \]
\end{proof}
\section{Case when both $f_1, f_2 \in \mathcal{Z}_{k+1}$}
\label{sec:CVinZ2}
Here we prove the convergence of double recurrence $k$-th degree polynomial Wiener Wintner averages for the case where $f_1, f_2 \in \mathcal{Z}_{k+1}$. To do so, we will use the structural properties of nilsystems discussed in \cite{HostKraCubes}. The summary of this result is provided in \cite{WWDR}.
\begin{theorem}\label{BothInZk+1}
Let $(X, \mathcal{F}, \mu, T)$ be an ergodic dynamical system, and $p$ be a $k$-th degree polynomial of real coefficients. Suppose $f_1, f_2 \in \mathcal{Z}_{k+1}$, and $a, b \in \ZZ$ such that $a \neq b$. Then the average
\begin{equation}\label{generalCV} \frac{1}{N} \sum_{n=1}^N f_1(T^{an}x)f_2(T^{bn}x)e^{2\pi i p(n) t} \end{equation}
converges, as $N \to \infty$, off of a single null-set that is independent of $t$.
\end{theorem}
We will break the proof of this theorem into two cases: First for the case when $t$ is irrational, and second for the case $t$ is rational. The proof for the case when $t$ is irrational is very similar to the proof of the case when $t$ is irrational from Theorem 7.3 in \cite{WWDR} (which is the case $p(n) = n$), which we shall discuss briefly. 

We first prove this for the case $f_1$ and $f_2$ are continuous. Let $Z_1$ be a Kronecker system of $X$. Recall that there exists $\beta \in Z_1$ so that $T$ acts as a rotation by $\beta$, and $\beta$ generates a dense cyclic subgroup in $Z_1$, which we denote $B$. We define a multiplicative character $\phi_t: B \to \mathbb{T}$ such that $\phi(\beta) = e^{2\pi i t}$; since $t$ is irrational, $e^{2\pi i t}$ generates a dense cyclic subgroup in $\mathbb{T}$. It can be shown that $\phi_t$ can be extended to $\overline{\phi}_t: Z_1 \to \mathbb{T}$ so that $\left. \overline{\phi}_t \right|_{B} = \phi_t$.

We define $f_t: Z_{k+1} \to \mathbb{T}$, where $Z_{k+1}$ is the Host-Kra-Ziegler system of order $k+1$, such that $f_t = \overline{\phi}_t \circ \pi$, where $\pi: Z_{k+1} \to Z_1$ is the factor map. We observe that $f_t(T^{p(n)}x) = f_t(x) e^{2\pi i p(n)t}$, and
\[ \frac{1}{N} \sum_{n=1}^N f_1(T^{an}x)f_2(T^{bn}x)f_t(T^{p(n)}x) = \frac{f_t(x)}{N} \sum_{n=1}^N f_1(T^{an}x)f_2(T^{bn}x) e^{2\pi i p(n)t}. \]
Note that the left-hand side of the equation above converges by Leibman's result on pointwise convergence of polynomial actions on nilsystems \cite{Leibman}, so our claim holds for the case $t$ is irrational, and $f_1$ and $f_2$ are both continuous.

For the sake of completeness, we show how to extend the claim to functions in $L^\infty(\mu)$; we use the following approximation argument to show that given $f_1, f_2 \in L^\infty(\mu) \cap \mathcal{Z}_{k+1}$, there exists a set of full measure such that the averages converge for all $t \in \RR$. Suppose $(f_{1'}^i)$ and $(f_{2'}^i)$ be sequences of continuous functions in $\mathcal{Z}_{k+1}$ that converge to $f_1$ and $f_2$ in $L^2$-norm respectively. Define  $$f_j^i(x) = \left\{ 
 \begin{array}{ll} \min(f_{j'}^i(x), \| f_j \|_\infty) & \mbox{if } f_{j'}^i(x) \geq 0,\\
 \max(f_{j'}^i(x), -\| f_j \|_\infty) & \mbox{if } f_{j'}^i(x) < 0,
\end{array} \right.$$ for $j = 1, 2$. Then $f_j^i \to f_j$ in $L^2(\mu)$ as $i \to \infty$, and $\|f_j^i\|_\infty \leq \| f_j \|_\infty$ for each $i$ and $j = 1, 2$. If we write $f_1 = (f_1 - f_1^i) + f_1^i$ and $f_2 = (f_2 - f_2^i) + f_2^i$, we see that
\begin{align}\label{4terms}
& W_N(f_1, f_2, x, p, t) = \frac{1}{N} \sum_{n=1}^N f_1(T^{an}x)f_2(T^{bn}x)e^{2\pi i p(n) t} \nonumber \\
&= \frac{1}{N} \sum_{n=1}^N (f_1 - f_1^i)(T^{an}x)(f_2 - f_2^i)(T^{bn}x)e^{2\pi i p(n) t} + \frac{1}{N} \sum_{n=1}^N(f_1 - f_1^i)(T^{an}x)f_2^i(T^{bn}x)e^{2\pi i p(n) t} \\
&+ \frac{1}{N} \sum_{n=1}^Nf_1^i(T^{an}x)(f_2 - f_2^i)(T^{bn}x)e^{2\pi i p(n) t} + \frac{1}{N} \sum_{n=1}^Nf_1^i(T^{an}x)f_2^i(T^{bn}x)e^{2\pi i p(n) t}. \nonumber
\end{align}
Note that the fourth term on the right hand side converges since $f_1^i$ and $f_2^i$ are both continuous. If we apply the Cauchy-Schwarz inequality on the second term on the right-hand side, we obtain
\[ \sup_{t \in \RR}\left|\frac{1}{N} \sum_{n=1}^N(f_1 - f_1^i)(T^{an}x)f_2^i(T^{bn}x)e^{2\pi i p(n) t} \right| \leq \left(  \frac{1}{N} \sum_{n=1}^N\left|(f_1 - f_1^i)\right|^2(T^{an}x) \right)^{1/2} \left( \frac{1}{N} \sum_{n=1}^N\left|f_2^i\right|^2(T^{bn}x) \right)^{1/2}. \]
Therefore,
\begin{align}\label{CS-Est} 
	& \sup_{N \geq 1, t \in \RR} \left|\frac{1}{N} \sum_{n=1}^N(f_1 - f_1^i)(T^{an}x)f_2^i(T^{bn}x)e^{2\pi i p(n) t} \right| \\
	& \nonumber \leq \sup_{N \geq 1} \left(  \frac{1}{N} \sum_{n=1}^N\left|(f_1 - f_1^i)\right|^2(T^{an}x) \right)^{1/2} \left( \frac{1}{N} \sum_{n=1}^N\left|f_2^i\right|^2(T^{bn}x) \right)^{1/2} 
\end{align}
Since $\| f_2^i \|_\infty \leq \| f_2 \|_\infty$ for each $i$, we see that
\[ \left( \frac{1}{N} \sum_{n=1}^N\left|f_2^i\right|^2(T^{bn}x) \right)^{1/2} \leq \| f_2 \|_\infty. \]
We take the integral on the both sides of the equation ($\ref{CS-Est}$), and use H\"{o}lder's inequality and the maximal ergodic theorem (cf. Theorem 1.8 of \cite{AssaniWWET}) to obtain the following estimate:
\begin{align}\label{secondTerm}
&\int \sup_{N \geq 1, t \in \RR}\left|\frac{1}{N} \sum_{n=1}^N(f_1 - f_1^i)(T^{an}x)f_2^i(T^{bn}x)e^{2\pi i p(n) t} \right| d\mu(x) \\
&\leq \int \left(\sup_{N}  \frac{1}{N} \sum_{n=1}^N\left|(f_1 - f_1^i)\right|^2(T^{an}x) \right)^{1/2} \left(\sup_{N} \frac{1}{N} \sum_{n=1}^N\left|f_2^i\right|^2(T^{bn}x) \right)^{1/2} d\mu(x) \nonumber\\
&\leq \|f_2\|_\infty  \left( \int \sup_{N} \frac{1}{N} \sum_{n=1}^N\left|f_1 - f_1^i\right|^2(T^{an}x) d\mu(x) \right)^{1/2}  \nonumber\\
&\leq 2 \|f_2\|_\infty \|f_1 - f_1^i \|_2. \nonumber
\end{align}
We can obtain similar inequalities for the first and the third terms of ($\ref{4terms}$): We have
\begin{equation}\label{firstTerm} \int \sup_{N \geq 1, t \in \RR}\left| \frac{1}{N} \sum_{n=1}^{N} (f_1-f_1^i)(T^{an}x)(f_2-f_2^i)(T^{bn}x)e^{2\pi i p(n)t} \right| d\mu(x) \leq 4\|f_1 - f_1^i\|_2 \|f_2 - f_2^i\|_2 \end{equation}
and
\begin{equation}\label{thirdTerm} \int \sup_{N \geq 1, t \in \RR}\left| \frac{1}{N} \sum_{n=1}^{N} f_1^i(T^{an}x)(f_2-f_2^i)(T^{bn}x)e^{2\pi i p(n)t} \right| d\mu(x) \leq 2\|f_1\|_\infty \|f_2 - f_2^i\|_2. \end{equation}
We note that
\begin{align*}
0 &\leq \sup_{t \in \RR} \left(\limsup_{N \to \infty} Re\left( W_N(f_1, f_2, x, p, t)\right) - \liminf_{N \to \infty} Re \left(W_N(f_1, f_2, x, p, t)\right) \right) \\
&\leq 2\liminf_{i \to \infty} \left( \sup_{N \geq 1, t \in \RR}\left| \frac{1}{N} \sum_{n=1}^N (f_1-f_1^i)(T^{an}x)(f_2 - f_2^i)(T^{bn}x) e^{2\pi i p(n) t}\right| \right. \\
&+ \sup_{N \geq 1, t \in \RR}\left| \frac{1}{N} \sum_{n=1}^N (f_1-f_1^i)(T^{an}x)f_2^i(T^{bn}x) e^{2\pi i p(n) t}\right| \\
&+ \left. \sup_{N \geq 1, t \in \RR}\left| \frac{1}{N} \sum_{n=1}^N f_1^i(T^{an}x)(f_2 - f_2^i)(T^{bn}x) e^{2\pi i p(n) t}\right| \right),
\end{align*}
Note that the fourth term of $(\ref{4terms})$ vanishes since $f_1^i$ and $f_2^i$ are both continuous and we already know the averages converge. We would like to show that for $\mu$-a.e. $x \in X$,
\begin{equation}\label{sup-inf} \sup_{t \in \RR} \left(\limsup_{N \to \infty} Re(W_N(f_1, f_2, x, p, t)) - \liminf_{N \to \infty} Re( W_N(f_1, f_2, x, p, t)) \right) \leq 0 \end{equation}
using Fatou's lemma and the estimates we obtained from $ (\ref{secondTerm}), (\ref{firstTerm})$, and $(\ref{thirdTerm})$. 
Indeed,
\begin{align}
& 2\int \liminf_{i \to \infty} \left( \sup_{N \geq 1, t \in \RR}\left| \frac{1}{N} \sum_{n=1}^N (f_1-f_1^i)(T^{an}x)(f_2 - f_2^i)(T^{bn}x) e^{2\pi i p(n) t}\right| \right. \label{LHS}\\
&+ \sup_{N \geq 1, t \in \RR}\left| \frac{1}{N} \sum_{n=1}^N (f_1-f_1^i)(T^{an}x)f_2^i(T^{bn}x) e^{2\pi i p(n) t}\right| \nonumber \\
&+ \left. \sup_{N \geq 1, t \in \RR}\left| \frac{1}{N} \sum_{n=1}^N f_1^i(T^{an}x)(f_2 - f_2^i)(T^{bn}x) e^{2\pi i p(n) t}\right| \right) d\mu(x) \nonumber\\
&\leq 2\liminf_{i \to \infty}\left( \int \sup_{N \geq 1, t \in \RR}\left| \frac{1}{N} \sum_{n=1}^N (f_1-f_1^i)(T^{an}x)(f_2 - f_2^i)(T^{bn}x) e^{2\pi i p(n) t}\right| d\mu(x)\right. \label{RHS}\\
&+\int \sup_{N \geq 1, t \in \RR}\left| \frac{1}{N} \sum_{n=1}^N (f_1-f_1^i)(T^{an}x)f_2^i(T^{bn}x) e^{2\pi i p(n) t}\right| d\mu(x) \nonumber\\
&+ \left. \int \sup_{N \geq 1, t \in \RR}\left| \frac{1}{N} \sum_{n=1}^N f_1^i(T^{an}x)(f_2 - f_2^i)(T^{bn}x) e^{2\pi i p(n) t}\right| d\mu(x) \right) = 0 \nonumber
\end{align}
since each integrals in $(\ref{RHS})$ is bounded by either $\|f_1 - f_1^i\|_2$, $\|f_2 - f_2^i\|_2$, or a product of both, which goes to $0$ as $i \to \infty$.  Since inside the integral of ($\ref{LHS}$) is nonnegative, we know that
\begin{align*}
&\liminf_{i \to \infty} \left(\sup_{N \geq 1, t \in \RR}\left| \frac{1}{N} \sum_{n=1}^N (f_1-f_1^i)(T^{an}x)(f_2 - f_2^i)(T^{bn}x) e^{2\pi i p(n) t}\right| \right. \\
&+ \sup_{N \geq 1, t \in \RR}\left| \frac{1}{N} \sum_{n=1}^N (f_1-f_1^i)(T^{an}x)f_2^i(T^{bn}x) e^{2\pi i p(n) t}\right| \\
&+ \left. \sup_{N \geq 1, t \in \RR}\left| \frac{1}{N} \sum_{n=1}^N f_1^i(T^{an}x)(f_2 - f_2^i)(T^{bn}x) e^{2\pi i p(n) t}\right| \right) = 0.
\end{align*}
Therefore, $(\ref{sup-inf})$ is established, and we use the similar argument to show that for $\mu$-a.e. $x \in X$,
\[ \sup_{t \in \RR} \left( \limsup_{N \to \infty} Im(W(f_1, f_2, x, p, t)) - \liminf_{N \to \infty} Im(W(f_1, f_2, x, p, t)) \right) = 0. \]
This shows that $W_N(f_1, f_2, x, p, t)$ converges for any $f_1, f_2 \in L^\infty(\mu)\cap \mathcal{Z}_{k+1}$ and $t$ irrational.

Hence it suffices to prove the case when $t$ is rational, and in fact, we only need to prove the convergence for a fixed $t \in \QQ$.

\begin{lemma}\label{rational}
To show that the average in ($\ref{generalCV}$) converges off a single null-set, it suffices to show that for each $t \in \QQ$, the averages in ($\ref{generalCV}$) converges for $\mu$-a.e. $x \in X$ .
\end{lemma}

\begin{proof}
We know that if $t$ is irrational, then ($\ref{generalCV}$) converges for all $x \in X$. Suppose $t$ is a fixed rational number, and we know that $V_t$ is a set of full measure such that ($\ref{generalCV}$) converges. Then
\[ V = \bigcap_{t \in \QQ} V_t \]
is a set of full measure since it is a countable intersection of sets of full measures. Hence, if $x \in V$, then $(\ref{generalCV})$ converges independent of $t \in \RR$.
\end{proof}

Before we present the proof for the case $t$ is rational, we first demonstrate a few examples on how to prove the Wiener-Winter convergence for polynomials of degree two and three. We believe these examples help the readers understand the proof for the general case.

The key ingredient of the proof for the case $t \in \QQ$ is the use of skew-product transformation on $\mathbb{T}^2$. First, we prove this for the case $p(n) = n^2$. We first find an appropriate Anzai skew product \cite{Anzai} on $\mathbb{T}^2$, and we apply this on the product of two functions on $X \times \mathbb{T}^2$, and take the averages. We know that this average converges a.e. by Bourgain's double recurrence theorem \cite{BourgainDR}, and the right hand side, after appropriate choice of constants and cancellations, will be the Wiener-Wintner averages for the case of polynomial $p(n) = n^2$. We apply Fubini's theorem and lemma $\ref{rational}$ to conclude the proof for $p(n) = n^2$.

We fix $t \in \QQ$ (in fact, this argument works for any $t \in \RR$). Consider the following skew-product transformation on $\mathbb{T}^2$:
\begin{equation}
R_\alpha (y, z) = (y + 2\alpha, z + y + \alpha), \end{equation}
where we will determine $\alpha$ in terms of $t$.  Since the transformations $y \mapsto y + \alpha$ and $z \mapsto z + y$ are both measure preserving with respect to the Lebesgue measure $m$ on $\mathbb{T}$, $R_\alpha$ is a measure-preserving skew-product transformation on $\mathbb{T}^2$ with respect to the product measure $m \otimes m$. We note that
\[ R_\alpha^n(y, z) = (y + 2n\alpha, z + ny + n^2\alpha). \]
Consider functions $F_j$ on $X \times \mathbb{T}^2$ to itself for $j = 1, 2$ such that
\[ F_j(x, y, z) = f_j(x)e^{2\pi i p_j y}e^{2\pi i q_j z}, \]
where $p_j$ and $q_j$ are constants to be determined. Let $U = T \times R_\alpha$. Then we compute
\begin{align}\label{n-sq}
& \nonumber \frac{1}{N} \sum_{n=1}^N F_1(U^{an}(x, y, z))F_2(U^{bn}(x, y, z)) \\
&= \frac{1}{N} \sum_{n=1}^N f_1(T^{an}x)f_2(T^{bn}x) e^{2\pi i p_1 (y + 2an\alpha)} e^{2\pi i p_2 (y + 2bn\alpha)}e^{2\pi i q_1 (z + any + a^2n^2\alpha)} e^{2\pi i q_1 (z + bny + b^2n^2\alpha)} \nonumber \\
&= \frac{e^{2\pi i (p_1 + p_2)y} e^{2\pi i (q_1 + q_2)z}}{N} \sum_{n=1}^N f_1(T^{an}x) f_2(T^{bn}x) e^{4\pi i (ap_1 + bp_2)n\alpha} e^{4\pi i (aq_1 + bq_2)ny}e^{2\pi i ( a^2q_1 + b^2q_2) n^2\alpha}.
\end{align}
If we set $C_{yz} = e^{2\pi i (p_1 + p_2)y} e^{2\pi i (q_1 + q_2)z}$, $aq_1 = -bq_2$ and $ap_1 = -bp_2$, then we obtain
\begin{align*}
& \frac{1}{N} \sum_{n=1}^N F_1(U^{an}(x, y, z))F_2(U^{bn}(x, y, z)) = \frac{C_{yz}}{N}\sum_{n=1}^N f_1(T^{an}x)f_2(T^{bn}x) e^{2\pi i b(b - a)q_2 n^2 \alpha}
\end{align*}

Note that the left hand side of the equation converges as a result of Bourgain's double recurrence theorem. Hence, the right hand side converges for $\mu \otimes m \otimes m$-a.e. $(x, y, z) \in X \times \mathbb{T}^2$. By applying Fubini's theorem and setting $\alpha =t/[b(b - a)q_2]$, we can conclude that
\begin{equation}\label{WWn-sq} \frac{1}{N}\sum_{n=1}^N f_1(T^{an}x)f_2(T^{bn}x) e^{2\pi i n^2 t} \end{equation}
converges for $\mu$-a.e. $x \in X$. Thus, there exists a set of full-measure $V_{t}$ such that for all $x \in V_{t}$, the averages in ($\ref{WWn-sq}$) converges. By applying lemma $\ref{rational}$, we obtain the desired result.

To show that the convergence holds for any second-degree polynomial (i.e. $p(n) = c_2 n^2 + c_1 n + c_0$), we apply the same skew-product that we used for the case $p(n) = n^2$, but we start with the averages
$$\frac{1}{N} \sum_{n=1}^N F_1(U^{an}(x,y,z)) F_2(U^{bn}(x,y,z)) e^{-2\pi i[(aq_1+ bq_2)ny - c_0\alpha]}$$ and we apply Theorem $\ref{WWDR}$ to prove the convergence. We let $ap_1 + bp_2 = c_1$ and $a^2q_1 + b^2q_2 = c_2$. Then we compute
\[ \frac{1}{N} \sum_{n=1}^N F_1(U^{an}(x, y, z))F_2(U^{bn}(x, y, z)) e^{-2\pi i [(aq_1 + bq_2) ny - c_0\alpha]} = \frac{C_{yz}}{N} \sum_{n=1}^N f_1(T^{an}x)f_2(T^{bn}x) e^{2\pi i (c_2n^2 + c_1n + c_0) \alpha}.
\]
Theorem 1.1 tells us that we can find a single null set in $X\times T^2$ off which the average 
$$\frac{1}{N} \sum_{n=1}^N F_1(U^{an}(x,y,z)) F_2(U^{bn}(x,y,z)) e^{2\pi ins}$$
converge for every s in R. In particular for $s = -(aq_1 + bq_2)y + c_0\alpha$, so the left hand side of this equation converges. Hence, after applying Fubini's theorem, the right-hand side of the equation converges for $\mu$-a.e. $x \in X$. By setting $\alpha = t$ and applying lemma $\ref{rational}$, we can conclude that the double recurrence Wiener Wintner result holds for a second degree polynomial.


Now we will prove this for the case $p(n) = n^3$. Again, as in the case of $p(n) = n^2$, we find an appropriate Anzai skew-product, and we apply this on functions $F_1$ and $F_2$ mentioned above. In order to show the convergence of the averages, we use the Wiener-Wintner result for the polynomial $p(n) = n^2$.

We consider the following skew product transformation on $\mathbb{T}^2$:
\[ R_{\alpha, 2} (y, z) = (y + 6\alpha, z + y^2 - \alpha^2). \]
Then we compute that
\[ R_{\alpha, 2}^n(y, z) = \left( y + 6n\alpha, z + ny^2 + 12n^2y\alpha - 18ny\alpha + 2n^3\alpha^2 - 3n^2\alpha^2 \right). \]
Let $U = T \times R_{\alpha, 2}$. If we compute the average 
\[\frac{1}{N} \sum_{n=1}^N F_1(U^{an}(x, y, z)) F_2(U^{bn}(x, y, z))\exp \left[ -12\pi i (b-a)(bq_2y\alpha - 3bq_2) \alpha^2 )n^2 \right], \]
we obtain Wiener-Wintner averages with the polynomial $p(n) = n^3$: First we compute that
\begin{align}\label{n-cubed}
& \frac{1}{N} \sum_{n=1}^N F_1(U^{an}(x, y, z)) F_2(U^{bn}(x, y, z)) \exp \left[ -12\pi i (b-a)(bq_2y\alpha - 3bq_2) \alpha^2 )n^2 \right] \nonumber \\
&= \frac{1}{N} \sum_{n=1}^N f_1(T^{an}x)f_2(T^{bn}x) \exp(2\pi i (p_1 + p_2)y) \exp(2\pi i (6ap_1 + 6bp_2)n\alpha)\nonumber \\
&\cdot \exp \left(2\pi i q_1 (z + any^2 + 6a^2n^2y\alpha - 6any\alpha + 12a^3n^3\alpha^2 - 18a^2n^2\alpha^2) \right) \nonumber \\
&\cdot \exp \left(2\pi i q_2 (z + bny^2 + 6b^2n^2y\alpha - 6bny\alpha + 12b^3n^3\alpha^2 - 18b^2n^2\alpha^2) \right) \\
&\cdot \exp \left[ -12\pi i (b-a)(bq_2y\alpha - 3bq_2) \alpha^2 )n^2 \right]\nonumber\\
&= \frac{C_{y, z}}{N} \sum_{n=1}^N f_1(T^{an}x)f_2(T^{bn}x) \exp(2\pi i (6ap_1 + 6bp_2)n\alpha) \exp(2\pi i (aq_1 + bq_2)ny)\nonumber \\
&\cdot \exp \left(2\pi i (6a^2q_1 + 6b^2q_2) n^2y\alpha \right)\exp \left(2\pi i (-6aq_1 - 6bq_2) ny\alpha \right) \nonumber \\
&\cdot \exp \left(2\pi i (12a^3q_1 + 12b^3q_2) n^3\alpha^2 \right)\exp \left(2\pi i (-18a^2q_1 - 18b^2q_2) n^2\alpha^2 \right) \nonumber\\
&\cdot \exp \left[ -12\pi i (b-a)(bq_2y\alpha - 3bq_2) \alpha^2 )n^2 \right], \nonumber
\end{align}
where we denoted $C_{y, z} = \exp(2\pi i (p_1 + p_2)y)\exp(2\pi i (q_1 + q_2)z)$. After setting $p_1 = -p_2$ and $aq_1 = -bq_2$,
we obtain
\begin{align} \label{n-cubed-CV}
& \frac{1}{N} \sum_{n=1}^N F_1(U^{an}(x, y, z)) F_2(U^{bn}(x, y, z))\exp \left[ -12\pi i (b-a)(bq_2y\alpha - 3bq_2) \alpha^2 n^2 \right] \nonumber \\
&= \frac{C_{y, z}}{N} \sum_{n=1}^N f_1(T^{an}x) f_2(T^{bn}x) \exp(2\pi i (12bq_2(b^2-a^2)) n^3 \alpha^2).
\end{align}
On the left-hand side, we set $s = (b-a)(bq_2y\alpha - 3bq_2) \alpha^2$, and apply ($\ref{WWn-sq}$), the previous Wiener-Winter result for the polynomial $p(n) = n^2$, to show that the left hand side converges for $\mu \otimes m \otimes m$-a.e. $(x, y, z) \in X \times \mathbb{T}^2$. Therefore, by applying Fubini's theorem, the right-hand side converges for $\mu$-a.e. $x \in X$. So in particular, if  We set $\alpha^2 = t/[12bq_2(b^2-a^2)]$, we have concluded that there exists a set of full measure $V_t$ such that for all $x \in V_t$, the average
\[ \frac{1}{N} \sum_{n=1}^N f_1(T^nx)f_2(T^{2n}x) \exp(2\pi i n^3 t)\]
converges. After we apply lemma $\ref{rational}$, we obtain the desired convergence result for $p(n) = n^3$.

Now we show that the result holds for any degree-three polynomial: $p(n)  = c_3n^3 + c_2n^2 + c_1n + c_0$. We change the exponential term on the left-hand side of $(\ref{n-cubed-CV})$ to 
$$\exp \left[ -12\pi i (b-a)(bq_2y\alpha - 3bq_2 + c_2) \alpha^2n^2 + c_1\alpha^2n + \alpha^2c_0 \right],$$
 and we compute
\begin{align}\label{deg3poly}
& \frac{1}{N} \sum_{n=1}^N F_1(U^{an}(x, y, z)) F_2(U^{bn}(x, y, z))\exp \left[ -12\pi i (b-a)(bq_2y\alpha - 3bq_2 + c_2) \alpha^2n^2 + c_1\alpha^2n + \alpha^2c_0 \right] \nonumber \\
&= \frac{C_{y, z}}{N} \sum_{n=1}^N f_1(T^nx) f_2(T^{2n}x) \exp(2\pi i [12bq_2(b^2-a^2) n^3 + c_2n^2 + c_1n + c_0]\alpha^2)
\end{align}
Note that the left hand side of the equation converges $\mu \otimes m^2$-a.e. $(x, y, z) \in X \times \mathbb{T}^2$ by the Wiener-Wintner result that we obtained for degree-two polynomials (in this case, $t = 1$). Hence, by applying Fubini's theorem, the right hand side converges for $\mu$-a.e. $x \in X$. Choose $q_2$ so that $c_3 = 12bq_2(b^2 - a^2)$ and $\alpha^2 = t$, then apply the similar argument using Fubini's theorem and lemma $\ref{rational}$ to obtain the result for degree-three polynomials.

Here we present the proof for any polynomial with real coefficients. The main idea behind the proof is similar to the case of polynomials of degree two and three.
\begin{proof}[Proof of Theorem $\ref{BothInZk+1}$ when $t$ is rational]
We will prove this case using induction. Let $t \in \QQ$ fixed. The base case $k = 1$ is proven in Theorem 7.3 of \cite{WWDR}. Assume the Wiener-Wintner result holds for the case $1 \leq k \leq m$. It suffices to show that the inductive step holds for the case $p(n) = n^{m+1}$. Indeed, consider the following skew-product transformation on $\mathbb{T}^2$:
\[ R_\alpha(y, z) = (y + \alpha, z + y^m). \]
Since the maps $y \mapsto y + \alpha$ and $z \mapsto z + y^m$ are both measure preserving on $\mathbb{T}$ (with respect to the Lebesgue probability measure $m$), we know that $R_\alpha$ is also measure-preserving on $\mathbb{T}^2$ by \cite{Anzai}. In fact, we can compute that
\[ R_\alpha^n(y, z) = \left(y + n\alpha, z + \sum_{l=0}^{n-1} (y + l\alpha)^m \right) = \left(y + n\alpha, z + ny^m + \Phi_{n-1}(y, \alpha) + \left( \sum_{l=0}^{n-1} l^m\right)\alpha^m \right), \]
where 
\[ \Phi_{n-1}(y, \alpha) = \sum_{q=1}^{m-1} {m \choose q} y^{m - q} \alpha^q \left( \sum_{l=0}^{n-1} l^q \right). \]
Note that, by Faulhaber's formula, the sum $\displaystyle \sum_{l=0}^{n-1} l^q$ is a polynomial of degree $q+1$ for $q = 1, \ldots, m$. In particular, we denote
\[ \phi(n) = \sum_{l=0}^{n-1} l^m \]
to be the $m+1$-th degree polynomial of $n$.

Consider functions $F_j$ on $X \times \mathbb{T}^2$ to itself for $j = 1, 2$ such that
\[ F_j(x, y, z) = f_j(x)\exp(2\pi i p_j y)\exp(2\pi i q_j z), \]
where $p_j$ and $q_j$ are constants to be determined. We note that $q_1 \phi(an) + q_2 \phi(bn)$ is a polynomial of degree $m+1$, so there exists a nonzero real constant $c_{m+1}$ and an $m$-th degree polynomial $\psi(n)$ such that
\[ q_1 \phi(an) + q_2 \phi(bn) = c_{m+1} n^{m+1} + \psi(n). \]
Let $U = T \times R_\alpha$. Then we compute
\begin{align*}
& \frac{1}{N} \sum_{n=1}^N F_1(U^{an}(x, y, z))F_2(U^{bn}(x, y, z))\exp(-2\pi i (q_1 \Phi_{an-1}(y, \alpha) + q_2 \Phi_{bn-1} (y, \alpha) + \psi(n))) \\
&= \frac{1}{N} \sum_{n=1}^N f_1(T^{an}x) f_2(T^{bn}x) \exp(2\pi i p_1(y + an\alpha))\exp(2\pi i p_2(y + bn\alpha))\\
&\cdot \exp(2\pi i q_1 (z + any^m + \Phi_{an-1}(y, \alpha) + \phi(an)\alpha^m))\exp(2\pi i q_2 (z + bny^m + \Phi_{bn-1}(y, \alpha) + \phi(bn)\alpha^m))\\
&\cdot \exp(-2\pi i (q_1 \Phi_{an-1}(y, \alpha) + q_2 \Phi_{bn-1} (y, \alpha) + \psi(n)))\\
&= \frac{C_{yz}}{N} \sum_{n=1}^N f_1(T^{an}x) f_2(T^{bn}x)\exp(2\pi i (ap_1 + bp_2)n\alpha)\exp(-2\pi i (q_1 \Phi_{an-1}(y, \alpha) + q_2 \Phi_{bn-1} (y, \alpha)))\\
&\cdot \exp(2\pi i (aq_1 + bq_2)ny^m) \exp(2\pi i (q_1 \Phi_{an-1}(y,\alpha) + q_2 \Phi_{bn-1} (y, \alpha))) \exp(2\pi i (q_1\phi(an) + q_2\phi(bn))\alpha^m) \\
&\cdot \exp(-2\pi i (q_1 \Phi_{an-1}(y, \alpha) + q_2 \Phi_{bn-1} (y, \alpha)+ \psi(n))),
\end{align*}
where $C_{y z} = \exp(2\pi i (p_1 + p_2)y)\exp(2\pi i (q_1 + q_2)z)$. If we set $p_1 = -2p_2$, $q_1 = -2q_2$, then we obtain
\begin{align*}
& \frac{1}{N} \sum_{n=1}^N F_1(U^{an}(x, y, z))F_2(U^{bn}(x, y, z)) \exp(-2\pi i (q_1 \Phi_{an-1}(y, \alpha) + q_2 \Phi_{bn-1} (y, \alpha)+ \psi(n))) \\
&= \frac{C_{y, z}}{N} \sum_{n=1}^N f_1(T^{an}x)f_2(T^{bn}x)\exp(2\pi i c_{m+1} n^{m+1}\alpha^m).
\end{align*}
Note that the left hand side converges $\mu \otimes m^2$-a.e. as $N \to \infty$ by the inductive hypothesis, so the right hand side converges $\mu \otimes m^2$-a.e. as well. By Fubini's theorem, the right hand side converges for $\mu$-a.e. $x \in X$. So if we set $\alpha^m  = t/c_{m+1}$, we have proved the Wiener-Wintner result for a polynomial $p(n) = n^{m+1}$ after applying lemma $\ref{rational}$. It is straightforward to show that the result holds for any $m+1$-degree polynomial.
\end{proof}
\bibliographystyle{plain}
\bibliography{WWDR_nil}

\end{document}